\documentclass[11pt]{article}
\usepackage{mathrsfs}
\usepackage{amsmath}
\usepackage{amsfonts}
\usepackage{amsmath,amsthm}
\usepackage{amsmath,amssymb,amsthm,latexsym}
\usepackage{amscd}
\usepackage[all]{xy}
\usepackage{hyperref}
\usepackage{fancyhdr}
\usepackage{indentfirst}
\usepackage{titlesec}
\usepackage{listings}
\usepackage{cases}
\usepackage{graphics}
\usepackage{graphicx}

\newtheorem{theorem}{Theorem}[section]

\newtheorem{lemma}[theorem]{Lemma}
\newtheorem{example}[theorem]{Example}

\theoremstyle{remark}
\newtheorem{remark}[theorem]{Remark}
\def\o{\omega}

\numberwithin{equation}{section}

\title{\textbf{M\"obius geometry of three dimensional Wintgen ideal submanifolds in $\mathbb{S}^{5}$}
\author { Zhenxiao Xie, Tongzhu Li, Xiang Ma, Changping Wang}}
\begin{document}
\maketitle
\begin{abstract}
Wintgen ideal submanifolds in space forms are those ones attaining equality at every point in the so-called DDVV inequality which relates the scalar curvature, the mean curvature and the normal scalar curvature.
This property is conformal invariant; hence we study them in the framework of M\"obius geometry, and restrict to three dimensional Wintgen ideal submanifolds in $\mathbb{S}^5$. In particular we give M\"obius characterizations for minimal ones among them, which are also known as (3-dimensional) austere submanifolds (in 5-dimensional space forms).
\end{abstract}

\vskip 1 cm
\section{Introduction}

The so-called DDVV inequality says that, given a $m$-dimensional submanifold $x:M^m\longrightarrow \mathbb{Q}^{m+p}(c)$ immersed in a real space form of dimension $m+p$ with constant sectional curvature $c$, at any point of $M$ we have
\begin{equation}\label{1.1}
s\leq c+||H||^2-s_N.
\end{equation}
Here $s=\frac{2}{m(m-1)}\sum_{1\leq i<j\leq n}\langle  R(e_i,e_j)e_j,e_i\rangle$ is the normalized scalar curvature with respect to the induced metric on $M$,
$H$ is the mean curvature,
and $s_N=\frac{2}{m(m-1)}||\mathbb{R}^{\perp}||$ is the normal scalar curvature. This remarkable inequality was first a conjecture due to De Smet, Dillen, Verstraelen and Vrancken \cite{Smet} in 1999, and proved by J. Ge, Z. Tang \cite{Ge} and Z. Lu \cite{Lu1} in 2008 independently.

As pointed out in \cite{Dajczer2}\cite{Smet}\cite{Lu1}\cite{Lu3},
it is a natural and important problem to characterize the extremal case,
i.e., those submanifolds attaining the equality \eqref{1.1} at every point, called \emph{Wintgen ideal submanifolds}.
In \cite{Ge} it was shown that the equality holds at
$x\in M^m$ if and only if there exist
an orthonormal basis $\{e_1,\cdots,e_m\}$ of $T_xM^m$ and
an orthonormal basis $\{n_1,\cdots,n_p\}$ of $T_x^{\bot}M^m$
such that the shape operators $\{A_{n_i},i=1,\cdots,m\}$ have the form
\begin{equation}\label{form1}
A_{n_1}=
\begin{pmatrix}
\lambda_1 & \mu_0 & 0 & \cdots & 0\\
\mu_0 & \lambda_1 & 0 & \cdots & 0\\
0  & 0 & \lambda_1 & \cdots & 0\\
\vdots & \vdots & \vdots & \ddots & \vdots\\
0  & 0 & 0 & \cdots & \lambda_1
\end{pmatrix},~~
A_{n_2}=
\begin{pmatrix}
\lambda_2\!+\!\mu_0 & 0 & 0 & \cdots & 0\\
0 & \lambda_2\!-\!\mu_0 & 0 & \cdots & 0\\
0  & 0 & \lambda_2 & \cdots & 0\\
\vdots & \vdots & \vdots & \ddots & \vdots\\
0  & 0 & 0 & \cdots & \lambda_2
\end{pmatrix},
\end{equation}
and
$$A_{n_3}=\lambda_3I_p,~~~~ A_{n_r}=0, r\ge 4.$$
This is the first step towards a complete classification.

Wintgen \cite{wint} first proved the inequality \eqref{1.1} for surfaces $M^2$ in $\mathbb{R}^4$, and that the equality holds if and only if the curvature ellipse of $M^2$ in $\mathbb{R}^4$ is a circle. Such surfaces are called \emph{super-conformal} surfaces. They come from projection of complex curves in the twistor space $\mathbb{C}P^3$ of $\mathbb{S}^4$ \cite{fb}. Together with totally umbilic submanifolds (spheres and planes), they provide the first examples of Wintgen ideal submanifolds. Note that they are not necessarily minimal surfaces in space forms. In particular, being super-conformal is a conformal invariant property, whereas being minimal is not.

The conformal invariance of Wintgen ideal property in the general case was pointed out in \cite{Dajczer1}. Thus it is appropriate to investigate and classify Wintgen ideal submanifolds under
the framework of M\"{o}bius geometry. For this purpose, the submanifold theory in M\"obius geometry established
by the fourth author will be briefly reviewed in Section~2.

We will always assume that the Wintgen ideal submanifolds in consideration are not totally umbilic.
Note that to have the shape operators taking the form in \eqref{form1}, the distribution
$\mathbb{D}=\mathrm{Span}\{e_1,e_2\}$ is well-defined. We call it
\emph{the canonical distribution}. The first M\"obius classification result was obtained by us in \cite{Li1}.\\

\noindent
{\bf Theorem A(Li-Ma-Wang\cite{Li1}):\hskip 3pt}
\emph{Let $x:M^m\to\mathbb{S}^{m+p}(m\geq3)$ be a Wintgen ideal submanifold and it is not totally umbilic. If the canonical distribution $\mathbb{D}=\mathrm{Span}\{e_1,e_2\}$ is integrable, then locally $x$ is M\"{o}bius equivalent to either one of the following three kinds of examples described in $\mathbb{R}^{m+p}$:\\
\indent (i) a cone over a minimal Wintgen ideal surface in ${\mathbb S}^{2+p}$;\\
\indent (ii) a cylinder over a minimal Wintgen ideal surface in ${\mathbb R}^{2+p}$;\\
\indent (iii) a rotational submanifold over a minimal Wintgen ideal surface in ${\mathbb H}^{2+p}$.}\\

In this paper we consider three dimensional Wintgen ideal submanifolds $x:M^3\to\mathbb{S}^5$ whose canonical distribution
$\mathbb{D}$ is not integrable.
There is a M\"obius invariant 1-form $\omega$ associated with $x$. For its definition as well as other basic equations and invariants, see Section~3.

Our main result is stated as below, which is proved in Section~4.\\

\noindent
{\bf Theorem B:\hskip 3pt}
\emph{Suppose $x:M^3\to\mathbb{S}^5$ is a Wintgen ideal submanifold whose canonical distribution $\mathbb{D}$ is not integrable. It is M\"obius equivalent to a minimal Wintgen ideal submanifold in a five dimensional space form $\mathbb{Q}^5(c)$
if and only if the 1-form $\omega$ is closed.}\\

Under some further conditions, in Section~5 we characterize minimal Wintgen ideal submanifolds coming from Hopf bundle over complex curves in $\mathbb{C}P^2$. We also discuss the classification of M\"obius homogeneous ones among Wintgen ideal 3-dimensional submanifolds in $\mathbb{S}^5$, which include the following example:
\[
x: \mathrm{SO}(3)~ \longrightarrow~\mathbb{S}^5,~~~
 (u, v, u\times v)  \mapsto \frac{1}{\sqrt{2}}(u, v).
\]

As to the geometric meaning of the 1-form $\omega$, we just mention that it could still be defined for Wintgen ideal submanifolds with dimension $m\ge 4$. In a forthcoming paper \cite{Li2} we will show that $d\omega=0$ is equivalent to the property that $\mathbb{D}=\mathrm{Span}\{e_1,e_2\}$ generates a 3-dimensional integrable distribution on $M^3$. Assume this is the case; then we will obtain a similar classification \cite{Li2} as in Theorem~A. These results again demonstrate the phenomenon described by our reduction theorem
\cite{Li0}.

To understand the classification result, it is necessary to note that among Wintgen ideal submanifolds,
there are a lot of minimal examples in space forms.
Although they do not exhaust all possible examples,
our classification demonstrates their importance as being
representatives in a M\"obius equivalence class of submanifolds,
or as building blocks of generic examples. Those minimal
Wintgen ideal surfaces are called \emph{super-minimal} in the previous literature, including examples like complex curves in $\mathbb{C}^n$ and minimal 2-spheres in $\mathbb{S}^n$.
For three dimensional submanifolds in 5-dimensional space forms $\mathbb{S}^5,\mathbb{R}^5,\mathbb{H}^5$, being minimal and Wintgen ideal is equivalent to being \emph{austere submanifolds}, i.e. the eigenvalues of the second fundamental form with respect to any normal direction occur in oppositely signed pairs.
Such submanifolds have been classified locally by Bryant \cite{br} for $M^3\to\mathbb{R}^5$ , by Dajczer and Florit \cite{Dajczer3} for $M^3\to\mathbb{S}^5$, and by
Choi and Lu \cite{Lu} for $M^3\to\mathbb{H}^5$.

Finally we note that in \cite{Dajczer1}, Dajczer and Florit have provided a parametric construction of Wintgen ideal submanifolds of codimension two and arbitrary dimension in terms of minimal surfaces in $\mathbb{R}^{m+2}$. Compared to our work, they
had no restriction on the dimension of $M$, and
the construction is explicit and valid for generic examples.
On the other hand, their descriptions were not in a M\"obius invariant language.
In another paper \cite{Li3}, we will give a construction of all Wintgen ideal submanifolds of codimension two and arbitrary dimension $m$ in terms of holomorphic, isotropic curves in a complex quadric $Q^{m+2}$. \\

\textbf{Acknowledgement} This work is funded by the Project 10901006 and 11171004 of National Natural Science Foundation of China. We thank Professor Zizhou Tang for pointing out the homogeneous embedding of $\mathrm{SO}(3)$ in $\mathbb{S}^5$ to us. We are grateful to the referees for their helpful suggestions.

 \vskip 1 cm
\section{Submanifold theory in M\"obius geometry}

In this section we briefly review the theory of submanifolds
in M\"obius geometry. For details we refer to \cite{CPWang}, \cite{liu}.

Recall that in the classical light-cone model, the light-like (space-like) directions in the Lorentz space $\mathbb{R}^{m+p+2}_1$ correspond to points (hyperspheres) in the round sphere $\mathbb{S}^{m+p}$, and the Lorentz orthogonal group correspond to conformal transformation group of $\mathbb{S}^{m+p}$. The Lorentz metric is written out explicitly as
\[
\langle  Y,Z\rangle=-Y_0Z_0+Y_1Z_1+\cdots+Y_{m+p+1}Z_{m+p+1},
\]
for
$Y=(Y_0,Y_1,\cdots,Y_{m+p+1}), Z=(Z_0,Z_1,\cdots,Z_{m+p+1})\in
\mathbb{R}^{m+p+2}_1$.

Let $x:M^m\rightarrow \mathbb{S}^{m+p}\subset \mathbb{R}^{m+p+1}$ be a submanifold without umbilics. Take $\{e_i|1\le i\le m\}$ as the tangent frame with respect to the induced metric $I=dx\cdot dx$, and $\{\theta_i\}$ as the dual 1-forms.
Let $\{n_{r}|1\le r\le p\}$ be an orthonormal frame for the
normal bundle. The second fundamental form and
the mean curvature of $x$ are
\begin{equation}\label{2.1}
II=\sum_{ij,r}h^{r}_{ij}\theta_i\otimes\theta_j
n_{r},~~H=\frac{1}{m}\sum_{j,r}h^{r}_{jj}n_{r}=\sum_{r}H^{r}n_{r},
\end{equation}
respectively. We define the M\"{o}bius position vector $Y:
M^m\rightarrow \mathbb{R}^{m+p+2}_1$ of $x$ by
\begin{equation}\label{2.2}
Y=\rho(1,x),~~~
~~\rho^2=\frac{m}{m-1}\left|II-\frac{1}{m} tr(II)I\right|^2~.
\end{equation}
$Y$ is called \emph{the canonical lift} of $x$ \cite{CPWang}.
Two submanifolds $x,\bar{x}: M^m\rightarrow \mathbb{S}^{m+p}$
are M\"{o}bius equivalent if there exists $T$ in the Lorentz group
$O(m+p+1,1)$ such that $\bar{Y}=YT.$
It follows immediately that
\begin{equation}
\mathrm{g}=\langle dY,dY\rangle=\rho^2 dx\cdot dx
\end{equation}
is a M\"{o}bius invariant, called the M\"{o}bius metric of $x$.

Let $\Delta$ be the Laplacian with respect to $\mathrm{g}$. Define
\begin{equation}
N=-\frac{1}{m}\Delta Y-\frac{1}{2m^2}
\langle \Delta Y,\Delta Y\rangle Y,
\end{equation}
which satisfies
\[
\langle Y,Y\rangle=0=\langle N,N\rangle, ~~
\langle N,Y\rangle=1~.
\]
Let $\{E_1,\cdots,E_m\}$ be a local orthonormal frame for $(M^m,\mathrm{g})$
with dual 1-forms $\{\omega_1,\cdots,\omega_m\}$. Write
$Y_j=E_j(Y)$. Then we have
\[
\langle  Y_j,Y\rangle =\langle  Y_j,N\rangle =0, ~\langle  Y_j,Y_k\rangle =\delta_{jk}, ~~1\leq j,k\leq m.
\]
We define
\[
\xi_r=(H^r,n_r+H^r x),~~~1\le r\le p.
\]
Then $\{\xi_{1},\cdots,\xi_p\}$ form the orthonormal frame of the
orthogonal complement of $\mathrm{Span}\{Y,N,Y_j|1\le j\le m\}$.
And $\{Y,N,Y_j,\xi_{r}\}$ is a moving frame in $\mathbb{R}^{m+p+2}_1$ along $M^m$.
\begin{remark}\label{rem-xi}
Geometrically, at one point $x$, $\xi_r$ corresponds to the unique sphere tangent to $M^m$ with normal vector $n_r$ and the same mean curvature $H^r=\langle \xi_r,g\rangle$ where $g=(1,\vec{0})$ is a constant time-like vector. We call $\{\xi_r\}_{r=1}^p$ \emph{the mean curvature spheres} of $M^m$.
\end{remark}
We fix the range of indices in this section as below: $1\leq
i,j,k\leq m; 1\leq r,s\leq p$. The structure equations are:
\begin{equation}\label{equation}
\begin{split}
&dY=\sum_i \omega_i Y_i,\\
&dN=\sum_{ij}A_{ij}\omega_i Y_j+\sum_{i,r} C^r_i\omega_i \xi_{r},\\
&d Y_i=-\sum_j A_{ij}\omega_j Y-\omega_i N+\sum_j\omega_{ij}Y_j
+\sum_{j} B^{r}_{ij}\omega_j \xi_{r},\\
&d \xi_{r}=-\sum_i C^{r}_i\omega_i Y-\sum_{i,j}\omega_i
B^{r}_{ij}Y_j +\sum_{s} \theta_{rs}\xi_{s},
\end{split}
\end{equation}
where $\omega_{ij}$ are the connection $1$-forms of the M\"{o}bius
metric $\mathrm{g}$, and $\theta_{rs}$ are the normal connection $1$-forms. The
tensors
\begin{equation}
{\bf A}=\sum_{i,j}A_{ij}\omega_i\otimes\omega_j,~~ {\bf
B}=\sum_{i,j,r}B^{r}_{ij}\omega_i\otimes\omega_j \xi_{r},~~
\Phi=\sum_{j,r}C^{r}_j\omega_j \xi_{r}
\end{equation}
are called the Blaschke tensor, the M\"{o}bius second fundamental
form and the M\"{o}bius form of $x$, respectively.
The covariant derivatives $A_{ij,k}, B^{r}_{ij,k}, C^{r}_{i,j}$ are defined as usual. For example,
\begin{eqnarray*}
&&\sum_j C^{r}_{i,j}\omega_j=d C^{r}_i+\sum_j C^{r}_j\omega_{ji}
+\sum_{s} C^{s}_i\theta_{sr},\\
&&\sum_k B^{r}_{ij,k}\omega_k=d B^{r}_{ij}+\sum_k
B^{r}_{ik}\omega_{kj} +\sum_k B^{r}_{kj}\omega_{ki}+\sum_{s}
B^{s}_{ij}\theta_{sr}.
\end{eqnarray*}
The integrability conditions are given as below:
\begin{eqnarray}
&&A_{ij,k}-A_{ik,j}=\sum_{r}(B^{r}_{ik}C^{r}_j
-B^{r}_{ij}C^{r}_k),\label{equa1}\\
&&C^{r}_{i,j}-C^{r}_{j,i}=\sum_k(B^{r}_{ik}A_{kj}
-B^{r}_{jk}A_{ki}),\label{equa2}\\
&&B^{r}_{ij,k}-B^{r}_{ik,j}=\delta_{ij}C^{r}_k
-\delta_{ik}C^{r}_j,\label{equa3}\\
&&R_{ijkl}=\sum_{r}(B^{r}_{ik}B^{r}_{jl}-B^{r}_{il}B^{r}_{jk}
+\delta_{ik}A_{jl}+\delta_{jl}A_{ik}
-\delta_{il}A_{jk}-\delta_{jk}A_{il}),\label{equa4}\\
&&R^{\perp}_{rs ij}=\sum_k
(B^{r}_{ik}B^{s}_{kj}-B^{s}_{ik}B^{r}_{kj}). \label{equa5}
\end{eqnarray}
Here $R_{ijkl}$ denote the curvature tensor of $\mathrm{g}$.
Other restrictions on tensor $\bf B$ are
\begin{equation}
\sum_j B^{r}_{jj}=0, ~~~\sum_{i,j,r}(B^{r}_{ij})^2=\frac{m-1}{m}. \label{equa7}
\end{equation}
All coefficients in the structure equations are determined by $\{\mathrm{g}, {\bf B}\}$
and the normal connection $\{\theta_{rs}\}$.
%We have the following Ricci identities $$B_{ij,kl}-B_{ij,lk}=\sum_mB_{mj}R_{mikl}+\sum_mB_{im}R_{mjkl}.$$
Coefficients of M\"{o}bius invariants and
the isometric invariants are related as below. (We omit the formula for $A_{ij}$ since it will not be used later.)
\begin{align}
B^{r}_{ij}&=\rho^{-1}(h^{r}_{ij}-H^{r}\delta_{ij}),\label{2.22}\\
C^{r}_i&=-\rho^{-2}[H^{r}_{,i}+\sum_j(h^{r}_{ij}
-H^{r}\delta_{ij})e_j(\ln\rho)]. \label{2.23}
\end{align}
\begin{remark}\label{rem-xi2}
For $x: M^3 \rightarrow \mathbb{R}^5$, the M\"obius position vector $Y: M^3\rightarrow \mathbb{R}^7_1$ and the mean curvature sphere $\{\xi_{1},\cdots,\xi_p\}$ are given by
\[
Y=\rho(\frac{1+|x|^2}{2}, \frac{1-|x|^2}{2}, x),
\]
\[
\xi_r=\left(\frac{1+|x|^2}{2}, \frac{1-|x|^2}{2}, x\right)H^r+(x\cdot n_r,-x\cdot n_r,n_r).
\]
Note that $H^r=\langle \xi_r,g\rangle$ where $g=(-1,1,\vec{0})$ is a constant light-like vector.
For $x: M^3 \rightarrow \mathbb{H}^5 \subset \mathbb{R}^6_1$ (the hyperboloid model of the hyperbolic space), the corresponding formulae are
\[
Y=\rho(x,1),~~~\xi_r=(n_r+H^rx,H^r),~~r=1,\cdots,p.
\]
In this case $H^r=\langle \xi_r,g\rangle$ where $g=(\vec{0},1)$ is a constant space-like vector.
The M\"{o}bius invariants are related to
the isometric invariants still by \eqref{2.22}$\sim$ \eqref{2.23}.
\end{remark}

\section{Three dimensional Wintgen ideal submanifolds in $\mathbb{S}^5$}

From now on, we assume $x: M^{3}\to \mathbb{S}^{5}$ to be a three dimensional Wintgen ideal submanifold without umbilic points. According to \eqref{form1} and \eqref{2.22},\eqref{equa7}, that means we can choose a suitable tangent and normal frame ($\{E_1,E_2,E_3\}$ and $\{\xi_1,\xi_2\}$) such that the M\"obius second fundamental form $\bf{B}$ takes the form
\begin{equation}\label{3.1}
B^{1}=
\begin{pmatrix}
0 & \mu & 0\\
\mu & 0 & 0\\
0  & 0 & 0
\end{pmatrix},~~~~~~
B^{2}=
\begin{pmatrix}
\mu & 0 & 0\\
0 & -\mu & 0\\
0  & 0 & 0
\end{pmatrix}, ~~~~\mu=\frac{1}{\sqrt{6}}.
\end{equation}
\begin{remark}\label{rem-transform}
The distribution $\mathbb{D}=\mathrm{Span}\{E_1,E_2\}$ is well-defined. The same is true for the vector field $E_3$ up to a sign, and this sign is fixed on a connected and orientable subset of $M^3$.
Notice that the tangent and normal frames still allow a simultaneous transformation
\begin{equation}\label{transform}
(\widetilde{E}_1,\widetilde{E}_2)=(E_1,E_2)
\begin{pmatrix} ~~\cos t &\sin t \\ -\sin t & \cos t\end{pmatrix},
~~(\widetilde{\xi}_1,\widetilde{\xi}_2)=(\xi_1,\xi_2)
\begin{pmatrix} \cos 2t &-\sin 2t \\ \sin 2t & ~~\cos 2t\end{pmatrix}
\end{equation}
if we fix the induced orientation on the tangent and normal bundles and require that $B^1,B^2$ still take the form \eqref{3.1}.
\end{remark}

First we compute the covariant derivatives of $B^{r}_{ij}$. From (\ref{3.1}) we get
\begin{equation}\label{bb1}
B^1_{33,i}=B^2_{33,i}=B^1_{12,i}=B^2_{11,i}=B^2_{22,i}=0,~~1\leq i\leq 3.
\end{equation}
Other derivatives are related with the connection 1-forms $\o_{ij}$ as below:
\begin{equation}\label{bb2}
\begin{aligned}
\omega_{23}&=\sum_i\frac{B^1_{13,i}}{\mu}\omega_i
=-\sum_i\frac{B^2_{23,i}}{\mu}\omega_i;\\ \omega_{13}&=\sum_i\frac{B^1_{23,i}}{\mu}\omega_i
=~~\sum_i\frac{B^2_{13,i}}{\mu}\omega_i;\\
2\omega_{12}+\theta_{12}&=\sum_i\frac{-B^1_{11,i}}{\mu}\omega_i
=\sum_i\frac{B^1_{22,i}}{\mu}\omega_i
=\sum_i\frac{B^2_{12,i}}{\mu}\omega_i.
\end{aligned}
\end{equation}
By \eqref{equa3} and \eqref{bb1} we know the symmetry property below,
\begin{equation}\label{bb3}
B^1_{13,2}=B^1_{23,1}=B^1_{12,3}=0,
~B^2_{13,2}=B^2_{23,1}=B^2_{12,3},
\end{equation}
where we have used $B^1_{12,3}=0$ by \eqref{bb1}.
From this fact and comparing coefficients in \eqref{bb2}, we obtain
\begin{equation}\label{bb4}
\mu\o_{13}(E_1)=B^2_{13,1}=B^1_{23,1}=0,~
\mu\o_{23}(E_2)=-B^2_{23,2}=B^1_{13,2}=0.
\end{equation}
Similarly we know the coefficients in the following three
equalities are equal to each other:
\begin{equation}\label{bb5}
\begin{aligned}
-\mu\o_{23}(E_1)&=-B^1_{13,1}=B^2_{23,1},\\
\mu\o_{13}(E_2)&=~~B^1_{23,2}=B^2_{13,2},\\
\mu(2\o_{12}+\theta_{12})(E_3)&=-B^1_{11,3}=B^1_{22,3}=B^2_{12,3}.
\end{aligned}
\end{equation}

Next we derive the M\"obius form, using
\eqref{equa3} and the information on $B^r_{ij,k}$:
\begin{equation}\label{cc1}
C^1_3=B^1_{22,3}-B^1_{23,2}=0,~C^2_3=B^2_{11,3}-B^2_{13,1}=0.
\end{equation}
The other coefficients $\{C^{r}_j\}$ are obtained similarly as
below:
\begin{equation}\label{cc2}
\begin{aligned}
&C^{1}_1=-B^{1}_{13,3}=-\mu\omega_{23}(E_3),
~~~~C^{2}_2=-B^{2}_{23,3}=~~\mu\omega_{23}(E_3), \\
&C^{1}_2=-B^{1}_{23,3}=-\mu\omega_{13}(E_3), ~~~~C^{2}_1=-B^{2}_{13,3}=-\mu\omega_{13}(E_3),\\
&C^{1}_1=B^{1}_{22,1}=~\mu(2\o_{12}+\theta_{12})(E_1)
=~B^{2}_{12,1}=-C^{2}_2, \\
&C^{1}_2=B^{1}_{11,2}=-\mu(2\o_{12}+\theta_{12})(E_2) =-B^{2}_{12,2}=~C^{2}_1.
\end{aligned}
\end{equation}
For simplicity we introduce the following notations:
\begin{equation}\label{UVL}
\begin{aligned}
&U=~~\o_{23}(E_3)=-\frac{C^1_1}{\mu}=\frac{C^2_2}{\mu},\\
&V=-\o_{13}(E_3)=~~\frac{C^1_2}{\mu}=\frac{C^2_1}{\mu},\\
&L=\o_{13}(E_2)=-\o_{23}(E_1)=-\frac{B^1_{11,3}}{\mu}.
\end{aligned}
\end{equation}
Then we summarize what we know about the connection 1-forms and
the covariant derivatives $B^r_{ij,k}$ as below:
\begin{gather}
\o_{13}=L\o_2-V\o_3,~~~\o_{23}=-L\o_1+U\o_3;\notag\\
2\o_{12}+\theta_{12}=-U \o_1-V\o_2+L\o_3.\label{omega}
\end{gather}
By \eqref{omega} we have
\begin{equation}\label{L}
d\o_3=\o_{31}\wedge \o_1+\o_{32}\wedge\o_2\equiv 2L\o_1\wedge \o_2 ~~~mod(\o_3).
\end{equation}
So the distribution $\mathbb{D}=\mathrm{Span}\{E_1, E_2\}$ is integrable if and only if $L=0$ identically.

For the information on the Blaschke tensor $\bf A$, we use \eqref{equa2}.
It requires to compute the covariant derivatives of $C^r_j$, which is quite straightforward:
\[
C^1_{1,i}=-C^2_{2,i}, ~C^1_{2,i}=C^2_{1,i},
~C^1_{3,2}=C^1_1 L=-\mu UL, ~C^1_{3,1}=-C^1_2 L=-\mu VL.
\]
Now it follows that
\begin{align}
\mu (A_{11}-A_{22})&=\sum_k(B^1_{2k}A_{k1}-B^1_{1k}A_{k2})
=C^1_{2,1}-C^1_{1,2},\label{A11}\\
2\mu A_{12}&=\sum_k(B^2_{1k}A_{k2}-B^2_{2k}A_{k1})
=C^2_{1,2}-C^2_{2,1}=C^1_{1,1}+C^1_{2,2},\label{A12}\\
\mu A_{13}&=\sum_k(B^1_{2k}A_{k3}-B^1_{3k}A_{k2})
=C^1_{2,3}-C^1_{3,2}=C^1_{2,3}+\mu UL,\label{A13}\\
\mu A_{23}&=\sum_k(B^1_{1k}A_{k3}-B^1_{3k}A_{k1})
=C^1_{1,3}-C^1_{3,1}=C^1_{1,3}+\mu VL.\label{A23}
\end{align}

Consider a new frame $\{Y,{\hat Y},\eta_1,\eta_2,\eta_3,
\xi_1,\xi_2\}$ in ${\mathbb R}^{7}_1$ along $M^3$ as below, whose geometric meaning will be clear later (see Theorem~\ref{thm-envelop} and Remark~\ref{rem-hatY}).
\begin{gather}
\eta_1=Y_1+\frac{C_2^1}{\mu}Y=Y_1+VY,~~
\eta_2=Y_2+\frac{C^1_1}{\mu}Y=Y_2-UY,~~
\eta_3=Y_3-\lambda Y;\label{eta}\\
{\hat Y}=N-\frac{1}{2}(U^2+V^2+\lambda^2)Y-VY_1+U Y_2+\lambda Y_3.\label{hatY}
\end{gather}
Here $\lambda\in C^{\infty}(M^3)$ is an arbitrarily given smooth function at the beginning. The new frame is orthonormal except that
\[
\langle Y,Y\rangle=0=\langle {\hat Y},{\hat Y}\rangle,~
\langle Y,{\hat Y}\rangle=1.
\]
By the original structure equations \eqref{equation} we get
\begin{align}
d\xi_{1}&=-\mu\o_2\eta_1-\mu\o_1\eta_2+\theta_{12}\xi_2,\label{3.7}\\
d\xi_{2}&=-\mu \o_1\eta_1+\mu\o_2\eta_2-\theta_{12}\xi_1,\label{3.8}\\
d\eta_1&=-{\hat\o}_1Y-\o_1{\hat
Y}+\sum_k\Omega_{1k}\eta_k+\mu\o_2\xi_1+\mu\o_1\xi_2,\label{3.3}\\
d\eta_2&=-{\hat\o}_2Y-\o_2{\hat
Y}+\sum_k\Omega_{2k}\eta_k+\mu\o_1\xi_1-\mu\o_2\xi_2,\label{3.4}\\
d\eta_3&=-{\hat\o}_3Y-\o_3{\hat Y}+\sum_k\Omega_{3k}\eta_k,\label{3.5}\\
dY&=\o Y+\o_1\eta_1+\o_2\eta_2+\o_3\eta_3,\label{3.9}\\
d{\hat Y}&=-\o {\hat Y}+{\hat\o}_1\eta_1+{\hat\o}_2\eta_2+{\hat\o}_3\eta_3, \label{3.10}
\end{align}
Note that \eqref{3.7} and \eqref{3.8} give the first motivation for the definition of $\eta_1,\eta_2$ in \eqref{eta}.

Differentiate \eqref{3.7} $\sim$ \eqref{3.10}. We get the following integrability equations:
\begin{align}
&d\o_1=\o\wedge\o_1+\Omega_{12}\wedge\o_2+\Omega_{13}\wedge\o_3; \label{3.12}\\ &d\o_2=\o\wedge\o_2-\Omega_{12}\wedge\o_1+\Omega_{23}\wedge\o_3; \label{3.13}\\
&d\o_3=\o\wedge\o_3+\Omega_{31}\wedge\o_1+\Omega_{32}\wedge\o_2; \label{3.21}\\
&d\o_1=-(\theta_{12}+\Omega_{12})\wedge\o_2; \label{3.14}\\
&d\o_2=~~(\theta_{12}+\Omega_{12})\wedge\o_1; \label{3.15}\\
&d{\hat\o}_1=-\o\wedge{\hat\o}_1+\Omega_{12}\wedge{\hat\o}_2
+\Omega_{13}\wedge{\hat\o}_3; \label{3.16}\\
&d{\hat\o}_2=-\o\wedge{\hat\o}_2-\Omega_{12}\wedge{\hat\o}_1
+\Omega_{23}\wedge{\hat\o}_3; \label{3.17}\\
&d{\hat\o}_3=-\o\wedge{\hat \o}_3+\Omega_{31}\wedge{\hat\o}_1
+\Omega_{32}\wedge{\hat\o}_2; \label{3.22}\\
&d\Omega_{12}=\Omega_{13}\wedge\Omega_{32}-\o_1\wedge{\hat\o}_2
-{\hat\o_1}\wedge\o_2+2\mu^2\o_1\wedge\o_2; \label{3.18}\\
&d\Omega_{13}=\Omega_{12}\wedge\Omega_{23}-\o_1\wedge{\hat\o}_3
-{\hat\o}_1\wedge\o_3; \label{3.19}\\
&d\Omega_{23}=-\Omega_{12}\wedge\Omega_{13}-\o_2\wedge{\hat\o}_3
-{\hat\o}_2\wedge\o_3; \label{3.20}\\
&\Omega_{13}\wedge\o_1=\Omega_{23}\wedge\o_2;\hskip 5pt \Omega_{13}\wedge\o_2=-\Omega_{23}\wedge\o_1;\label{3.23}\\
&d\theta_{12}=2\mu^2\o_1\wedge\o_2;\label{3.24}\\
&\o_1\wedge{\hat\o_2}=-\o_2\wedge{\hat\o}_1;\hskip 5pt \o_1\wedge{\hat\o_1}=\o_2\wedge{\hat\o}_2;\label{3.25}\\
&d\o=-\o_1\wedge{\hat\o}_1-\o_2\wedge{\hat\o}_2
-\o_3\wedge{\hat\o}_3.\label{3.26}
\end{align}
The 1-forms $\omega,{\hat\o}_i,\Omega_{ij}=-\Omega_{ji}$
are determined by \eqref{equation} and \eqref{eta},\eqref{hatY}:
\begin{align}
\Omega_{12}&=\langle d\eta_1,\eta_2\rangle
=\o_{12}+U\o_1+V\o_2,\label{Omega12}\\
\Omega_{13}&=\langle d\eta_1,\eta_3\rangle
=\lambda\o_1+L\o_2,\label{Omega13}\\
\Omega_{23}&=\langle d\eta_2,\eta_3\rangle
=-L\o_1+\lambda\o_2,\label{Omega23}\\
\o&=\langle dY,{\hat Y}
\rangle=-V\o_1+U\o_2+\lambda\o_3.
\end{align}
It follows from \eqref{3.25} that there exist some functions
$\hat{F},\hat{G}$ such that
\begin{equation}
{\hat\o_1}=\hat{F}\o_1+\hat{G}\o_2,~~ {\hat\o_2}=-\hat{G}\o_1+\hat{F}\o_2.\label{FG}
\end{equation}
A straightforward yet lengthy computation find
\begin{align*}
{\hat\o}_1&=\langle d{\hat Y},\eta_1\rangle \\
&= d\left(N-\frac{1}{2}(U^2+V^2+\lambda^2)Y-VY_1+U Y_2+\lambda Y_3\right)\cdot (Y_1+VY)\\
&=\sum_i A_{i1}\o_i-\frac{1}{2}(U^2+V^2+\lambda^2)\o_1\\
&~~~~-dV
+V^2\o_1+U\o_{21}-UV\o_2+\lambda\o_{31}-\lambda V\o_3\\
&=\left(A_{11}+\frac{1}{2}(U^2+V^2-\lambda^2)
-\frac{1}{\mu}C^1_{2,1}\right)\o_1
+\left(A_{12}-\frac{1}{\mu}C^1_{2,2}-\lambda L\right)\o_2,
\end{align*}
where we have used $d C^1_2=\sum_i C^1_{2,i}\o_i-C^1_1\o_{12}-C^2_2\theta_{21}$
and \eqref{omega}, \eqref{A13}.
For ${\hat\o}_2$ we compute in a similar manner, using
\eqref{A11},\eqref{A12} and \eqref{A23} to verify
\eqref{FG}, with the result as below:
\begin{align}
\hat{F}&=A_{11}+\frac{1}{2}(U^2+V^2-\lambda^2)
-\frac{1}{\mu}C^1_{2,1},\label{hatF}\\
\hat{G}&=A_{12}-\frac{1}{\mu}C^1_{2,2}-\lambda L
=\frac{C^1_{1,1}-C^1_{2,2}}{2\mu}-\lambda L.\label{hatG}
\end{align}
In particular we have the following observation.
\begin{lemma}\label{lem-G}
$\hat{G}$=0 if and only if $\lambda=\frac{G}{L}$ with $L=-\frac{B^1_{11,3}}{\mu}, G=\frac{C^1_{1,1}-C^1_{2,2}}{2\mu}$.
\end{lemma}

In the end of this section we make the following important geometric observation.\\

The spacelike 2-plane $\mathrm{Span}_{\mathbb{R}}\{\xi_1,\xi_2\}$ at
$p\in M^3$ is well-defined, and we call it
\emph{the 3-dimensional mean curvature sphere},
because it defines a 3-sphere tangent to $M^3$ at $p$ with the same mean curvature vector. (Please compare to Remark~\ref{rem-xi}.)

The first key observation is that in the codimension two case, $\mathrm{Span}_{\mathbb{R}}\{\xi_1,\xi_2\}$ is determined by
the complex line $\mathrm{Span}_{\mathbb{C}}\{\xi_{1}-i\xi_{2}\}$ and vice versa. So $\xi_1-i\xi_2\in \mathbb{C}^7$ represents the same geometric object.
It is a null vector with respect to the $\mathbb{C}$-linear extension of the Lorentz metric to $\mathbb{R}^7_1\otimes\mathbb{C}$.
The complex line spanned by it corresponds to a point
\[
[\xi_{1}-i\xi_{2}] \in Q^5=\{[Z]\in \mathbb{C}P^6| \langle Z,Z\rangle =0\}.
\]
It is similar to the conformal Gauss map of a (Willmore) surface
\cite{br0} and to the Gauss map of a hypersurface in $\mathbb{S}^n$ \cite{MaOhnita}.

The second key observation is that under the hypothesis of being
Wintgen ideal, this 3-sphere congruence is indeed a 2-parameter
family, and its envelope not only recovers $M^3$, but also
extends it to a 3-manifold as a circle bundle over a Riemann surface $\overline{M}$ (a holomorphic curve).
The underlying surface $\overline{M}$ comes from
the quotient surface $\overline{M}=M^3/\Gamma$ (at least locally) where $\Gamma$ is the foliation of $M^3$ by the integral curves of the vector fields $E_3$. More precise statement is as below.

\begin{theorem}\label{thm-envelop}
For a Wintgen ideal submanifold $x:M^3\to \mathbb{S}^5$ we have:

(1) The complex vector-valued function $\xi_{1}-i\xi_{2}$
locally defines a complex curve
\[
[\xi_1-i\xi_2]:\overline{M}=M^3/\Gamma\to Q^5\subset\mathbb{C}P^6.
\]
\indent (2) \emph{The 3-dimensional mean curvature spheres} $\mathrm{Span}\{\xi_1(p),\xi_2(p)\}$ is a two-parameter family of 3-spheres in $\mathbb{S}^5$ when $p$ runs through $M^3$.

(3) The Lorentz 3-space $\mathrm{Span}\{Y,{\hat Y},\eta_3\}$ and the two light-like directions $[Y],[{\hat Y}]\in \mathbb{R}P^6$ correspond to a circle and two points on it.
These circles are a two-parameter family. They foliate a three dimensional submanifold $\widehat{M}^3$ enveloped by the 3-dimensional mean curvature spheres
$\mathrm{Span}\{\xi_{1},\xi_{2}\}$, which includes $M^3$ as part of it.
On $M^3$ these circular arcs are indeed the integral curves
of the vector field $E_3$.

(4) This envelope $\widehat{M}^3\subset\mathbb{S}^5$,
as a natural extension of $x:M^3\to\mathbb{S}^5$,
is still a Wintgen ideal submanifold (at its regular points).
\end{theorem}
\begin{proof}
The structure equations \eqref{3.7} and \eqref{3.8} imply
\begin{equation}\label{J}
d(\xi_{1}-i\xi_{2})=i\mu(\o_1+i\o_2)(\eta_1+i\eta_2)
+i\theta_{12}(\xi_1-i\xi_2).
\end{equation}
Geometrically that means $\xi=[\xi_{1}-i\xi_{2}]:M^3\to \mathbb{C}P^6$ decomposes as a quotient map $\pi:M^3\to \overline{M}=M^3/\Gamma$ composed with a holomorphic immersion $\bar\xi:\overline{M}\to \mathbb{C}P^6$. Thus conclusion (1) is proved, and (2) follows directly.

To prove (3), notice that the light-like directions in $\mathrm{Span}\{Y,{\hat Y},\eta_3\}$ represent points on a circle.
Since $\{\xi_1,\xi_2,d\xi_1,d\xi_2\}$ span a 4-dimensional spacelike subspace by \eqref{J}, the corresponding 2-parameter family of 3-dimensional mean curvature sphere congruence has an
envelop $\widehat{M}$, whose points correspond to the light-like directions in the orthogonal complement $\mathrm{Span}\{Y,{\hat Y},\eta_3\}$. In particular $[Y],[{\hat Y}]$ are two points on this circle. Such circles form a 2-parameter family, with $\overline{M}$ as the parameter space. They give a foliation of $\widehat{M}^3$ which is also a circle fibration.

We assert that every integral curve $\gamma$ of $E_3$ is contained in such a circle. Because $Span\{\xi_1,\xi_2,\eta_1,\eta_2\}(p)$ is a fixed
subspace along an integral curve of $E_3$ passing $p\in M^3$ by \eqref{3.7}$\sim$\eqref{3.4} and \eqref{FG},\eqref{Omega13},\eqref{Omega23}. Thus the integration of $Y$ along $E_3$ direction is always located in the orthogonal complement $\mathrm{Span}\{Y(p),{\hat Y}(p),\eta_3(p)\}$, which describes a circle as above.
Thus $\widehat{M}^3\supset M^3$, and each circle fiber cover an integral curve of $E_3$. This verifies (3).

To prove (4), we need only to show that for arbitrarily chosen smooth function $\lambda:M^3\to \mathbb{R}$, the corresponding $[\hat{Y}]:M^3\to \mathbb{S}^5$ is a Wintgen ideal submanifold.
This is because at one point $p\in M^3$, when $\lambda$ runs over $(-\infty,\infty)$, $[\hat{Y}(p)]$ given by \eqref{hatY} will
cover the circle fiber except $[Y(p)]$ itself;
and when $\lambda: M^3\to \mathbb{R}$ is arbitrary, all
such local mappings will cover $\widehat{M}$ by their images.

We have to compute out the Laplacian of $\hat{Y}$
with respect to its induced metric $\hat\o_1^2+\hat\o_2^2+\hat\o_3^2$ which is necessary to
determine the normal frames (the mean curvature spheres) $\{\hat\xi_r\}$ of $\hat{Y}$.
The main difficulty is that the map $\hat{Y}:M^3\to \mathbb{R}^7_1$ is generally not conformal to $Y$ . Fortunately we need only to find an orthogonal frame
$\{\hat{E}_j\}_{j=1}^3$ with the same length
for $\hat{Y}$, and then using the fact $
\mathrm{Span}_{\mathbb{R}}\{\hat{Y},\hat{Y}_j,\sum_{j=1}^3 \hat{E}_j\hat{E}_j(\hat{Y})\}
=\mathrm{Span}_{\mathbb{R}}\{\hat{Y},\hat{Y}_j, \hat{\Delta}\hat{Y}\}.$\\

\textbf{Claim:} $\hat{Y}$ shares the same mean curvature spheres
$\{\xi_1,\xi_2\}$ as $Y$.\\

This requires to show $\langle\sum_{j=1}^3 \hat{E}_j\hat{E}_j(\hat{Y}),\xi_r\rangle=0.$
Since $0=\langle\hat{Y},\xi_r\rangle=\langle d\hat{Y},\xi_r\rangle=\langle \hat{Y},d\xi_r\rangle$ by \eqref{3.7}, \eqref{3.8} and \eqref{3.10}, we need only to verify
\[
\langle \hat{Y},\sum_{j=1}^3 \hat{E}_j\hat{E}_j(\xi_r)\rangle=0,~~r=1,2.
\]
For this purpose, suppose (keeping \eqref{FG} in mind):
\[
{\hat\o_1}=\hat{F}\o_1+\hat{G}\o_2,~~ {\hat\o_2}=-\hat{G}\o_1+\hat{F}\o_2,~~
\hat\o_3=a\o_1+b\o_2+c\o_3.
\]
Notice that we can always assume $\hat{Y}$ to be an immersion
at the points where $\widehat{M}$ is regular, hence $c\ne 0$.
Then one can take
\[
\hat{E}_1 = \hat{F} E_1+\hat{G} E_2 + a_{13}E_3,~
\hat{E}_2 = -\hat{G} E_1+\hat{F} E_2 +a_{23}E_3,~
\hat{E}_3 = a_{33}E_3,
\]
where $a_{13},a_{23},a_{33}$ are uniquely determined by
$\hat\o_i(\hat{E}_j)=(\hat{F}^2+\hat{G}^2)\delta_{ij}$.
The explicit form of $a_{13},a_{23},a_{33}$ is not important,
because when we insert the formulae above to $\sum_{j=1}^3 \hat{E}_j\hat{E}_j(\xi_r)$,
the terms involving $E_3$ will always be orthogonal to $\hat{Y}$. For example, $\langle E_3(\eta_1),\hat{Y}\rangle=-\langle \eta_1,E_3(\hat{Y})\rangle=\hat\o_1(E_3)=0$ by \eqref{FG}.
Thus we need only to compute the effect on $\xi_r$ of the operator below:
\[
(\hat{F} E_1+\hat{G} E_2)^2+(-\hat{G} E_1+\hat{F} E_2)^2
\thickapprox(\hat{F}^2+\hat{G}^2)\partial\bar{\partial}.
\]
The two sides are equal up to first order differential operators like $[E_1,E_2],E_1,E_2$, whose action on $\xi_r$ must be orthogonal to $\hat{Y}$; the complex differential operators are defined as usual:
\[
\partial=E_1-iE_2,~\bar\partial=E_1+iE_2.
\]
Since $(\o_1+i\o_2)(\bar\partial)=0$,
it follows from \eqref{J} that
\begin{align*}
\bar{\partial}(\xi_1-i\xi_2)&=i\theta_{12}(\bar{\partial})
(\xi_1-i\xi_2),\\
\partial\bar{\partial}(\xi_1-i\xi_2)&\in
\mathrm{Span}_{\mathbb{C}}\{\xi_1-i\xi_2,\eta_1+i\eta_2\}~~\bot~~ \hat{Y}.
\end{align*}
This completes the proof of the previous claim.

For $\hat{Y}$ we still take its canonical lift, whose derivatives are combinations of $\hat{Y},\eta_1,\eta_2,\eta_3$; its normal frame is still $\{\xi_1,\xi_2\}$. We read from \eqref{3.7}\eqref{3.8} that its M\"obius second fundamental form still take the same form as \eqref{3.1}. Thus conclusion (4) and the whole theorem is proved.
\end{proof}
\begin{remark}
In the proof above, among the integrability equations from \eqref{3.12}$\sim$\eqref{3.26}, only \eqref{3.25} and \eqref{3.23}
are necessary for us (to deduce the algebraic form
\eqref{FG} and \eqref{Omega13},\eqref{Omega23}, where the explicit  coefficients are not important). This is somewhat striking to the authors that the strong conclusion (4) follows from so few conditions. In a forthcoming paper \cite{Li3} we will give a general treatment of codimension-two Wintgen ideal submanifolds based on the observations in this theorem.

Another interesting feature is the resemblance between  conclusion (4) and the duality theorem for Willmore surfaces
\cite{br0}.
\end{remark}

\section{Minimal Wintgen ideal submanifolds}

Since we have classified all Wintgen ideal submanifolds in \cite{Li1} whose canonical distribution $\mathbb{D}=\mathrm{Span}\{E_1,E_2\}$
is integrable (that means $L=0$), in the rest of this paper we will only consider the case
\[
L\neq0.
\]
From now on we take the frame \eqref{eta} and \eqref{hatY}, and make
the following
\begin{equation}\label{lambda}
\textbf{Assumption:}~~~~
\lambda=\frac{G}{L}=\frac{C^1_{1,1}-C^1_{2,2}}{-2 B^1_{11,3}}~.\qquad\qquad\qquad\qquad\qquad
\end{equation}
By Lemma~\ref{lem-G} we have
\begin{equation}\label{3.27}
{\hat\o_1}=\hat{F}\o_1,~~ {\hat\o_2}=\hat{F}\o_2.
\end{equation}

\begin{remark}\label{rem-hatY}
The correspondence $Y\to \hat{Y}$ describes a
self-mapping of the enveloping submanifold $\widehat{M}^3$
where $\hat{Y}$ and $Y$ are located on the same circle fiber.
At corresponding points they share the same normal vector
fields $\{\xi_r\}$, with respect to which we can talk about
principal directions.
According to Theorem~\ref{thm-envelop} and Lemma~\ref{lem-G},
the correspondence $Y\to \hat{Y}$ preserves the principal directions for any $\xi_r$ if and only if $\lambda=G/L$.
This explains the geometric significance of the condition \eqref{lambda}.
\end{remark}

Under the assumption \eqref{lambda},
\begin{equation}\label{o}
\o=-V\o_1+U\o_2+\frac{G}{L}\o_3.
\end{equation}
Together with \eqref{3.26} and \eqref{3.27} there must be
\begin{equation}
d\o=-\o_3\wedge{\hat\o}_3.\label{3.28}
\end{equation}
On the other hand, it follows from \eqref{omega} and $\lambda=G/L$ that \eqref{Omega12},\eqref{Omega13} and \eqref{Omega23} now take the form
\begin{gather}
2\Omega_{12}+\theta_{12}=U\o_1+V\o_2+L\o_3,\label{Omega12+}\\
\Omega_{13}=\frac{G}{L}\o_1+L\o_2,~~
\Omega_{23}=-L\o_1+\frac{G}{L}\o_2.\label{Omega13+}
\end{gather}
Insert these into structure equations \eqref{3.19} and \eqref{3.20}, and simplify by \eqref{3.14}, \eqref{3.15} together with \eqref{3.27}. We have
\begin{align*}
dL\wedge\o_2&=
(U\o_1+V\o_2+L\o_3)\wedge(-L\o_1+\frac{G}{L}\o_2)
-\o_1\wedge({\hat \o}_3-d\frac{G}{L})-\hat{F}\o_1\wedge\o_3,\\
dL\wedge\o_1&=
(U\o_1+V\o_2+L\o_3)\wedge(\frac{G}{L}\o_1+L\o_2)
+\o_2\wedge({\hat \o}_3-d\frac{G}{L})+\hat{F}\o_2\wedge\o_3.
\end{align*}
Comparing the coefficients of $\o_1\wedge\o_2,\o_1\wedge\o_3,\o_2\wedge\o_3$ in these two
equations separately, we obtain
\begin{align}
\hat\o_3(E_1)&=E_2(L)+UL+E_1(\frac{G}{L})-V\frac{G}{L},\label{compare1}\\
\hat\o_3(E_2)&=-E_1(L)+VL+E_2(\frac{G}{L})+U\frac{G}{L},\label{compare2}\\
\hat\o_3(E_3)&=E_3(\frac{G}{L})+L^2-\hat{F},\label{compare3}\\
E_3(L)&=G.\label{compare4}
\end{align}
Similarly, inserting \eqref{3.27} into \eqref{3.16} and \eqref{3.17} yields
\begin{align*}
d\hat{F}\wedge\o_1&=
(2\Omega_{12}+\theta_{12})\wedge \hat{F}\o_2
-\o\wedge \hat{F}\o_1+\Omega_{13}\wedge\hat\o_3,\\
d\hat{F}\wedge\o_2&=
-(2\Omega_{12}+\theta_{12})\wedge \hat{F}\o_1
-\o\wedge \hat{F}\o_2+\Omega_{23}\wedge\hat\o_3.
\end{align*}
Invoking \eqref{o},\eqref{Omega12+},\eqref{Omega13+} and comparing the coefficients of $\o_1\wedge\o_2,\o_1\wedge\o_3,\o_2\wedge\o_3$ in these two
equations separately, one gets
\begin{align}
\hat{F}&=\hat\o_3(E_3),\label{compare5}\\
E_1(\hat{F})
&=2V\hat{F}-\frac{G}{L}\hat\o_3(E_1)-L\hat\o_3(E_2),\label{compare6}\\
E_2(\hat{F})
&=-2U\hat{F}+L\hat\o_3(E_1)-\frac{G}{L}\hat\o_3(E_2),\label{compare7}\\
E_3(\hat{F})
&=-\frac{G}{L}(\hat{F}+\hat\o_3(E_3))=-2\frac{G}{L}\cdot\hat{F}.\label{compare8}
\end{align}

\begin{remark}\label{rem-invariant}
We point out that $L,G,\o,\hat{F}$ are well-defined M\"obius invariants.
It is necessary and sufficient to verify that they are independent to the choice of the frames, or equivalently, that they are invariant under the transformation \eqref{transform}.
This is obvious for $L$ by \eqref{L}, for $G$ by \eqref{compare4}, and for $\o$ by $\o=\langle dY,{\hat Y}\rangle$ where the frame vectors $\hat{Y},\eta_3$ are now canonically chosen after taking $\lambda=\frac{G}{L}$.
For $\hat{F}$ we can verify the invariance under the transformation \eqref{transform} by $\hat{F}=\langle E_1({\hat Y}),\eta_1\rangle$, or just using \eqref{compare5}.

On the other hand, $U,V$ correspond to $\{C^r_i\}$, components  of the M\"obius form, which depend on the choice of $\{E_1,E_2\}$ and $\{\xi_1,\xi_2\}$. Yet in that case we can choose the angle
$t$ in \eqref{transform} suitably such that $V=0$ identically.
Then the new function $U$ is well-defined and M\"obius invariant. (There are other choice of the frame in a canonical way, and any of them works in the proof to Theorem~\ref{thm-homog} later.)
\end{remark}

Now we can state our M\"{o}bius characterization theorem for minimal Wintgen ideal submanifolds of dimension three in five dimensional space form.

\begin{theorem}\label{thm-minimal}
Let $x: M^3\to {\mathbb S}^{5}$ be
a Wintgen ideal submanifold. Assume the distribution $\mathbb{D}=\mathrm{Span}\{E_1, E_2\}$ to be non-integrable, i.e., $L\ne 0$.
Then the following conditions are equivalent:

(1) $ d\o=0,$ where $\o=-V\o_1+U\o_2+\frac{G}{L}\o_3$.

(2) The correspondence $Y\to \hat{Y}$ of the enveloping submanifold $\widehat{M}^3$ is a conformal map. ($\hat{Y}$ might be degenerate.)

(3) $x: M^3\to {\mathbb S}^{5}$ is M\"obius equivalent to a
minimal Wintgen ideal submanifold in a space form.
In particular, the space form is $\mathbb{S}^5$,$\mathbb{R}^5$
or $\mathbb{H}^5$ depending on whether $\hat{F}$ is positive, zero or negative.
\end{theorem}
\begin{proof}
If $Y$ and $\hat Y$ are conformal, then there exists a non-negative function $a$ so that
\begin{equation}
\hat\o_1^2+\hat\o_2^2+\hat\o_3^2=a(\o_1^2+\o_2^2+\o_3^2).\label{3.51}
\end{equation}
Combining with \eqref{3.27} and \eqref{3.28} we have $d\o=0$.
Thus (2) implies (1).

Next we show (1) implies (2) and (3). If $d\o=0$, it follows from \eqref{3.28}\eqref{compare5} that $\hat \o_3=\hat{F}\o_3$. Together with \eqref{3.27}, $\hat \o_j=\hat{F}\o_j$ for $j=1,2,3$. So $Y$ and $\hat Y$ are conformal, and (2) is proved.

Using \eqref{compare6}$\sim$\eqref{compare8} we get
\begin{equation}
d\hat{F}+2\hat{F}\omega = 0.\label{3.50}
\end{equation}
Now the structure equations can be rewritten as below:
\begin{align}
d(\hat{F}Y+{\hat Y})&=-\o (\hat{F}Y+{\hat Y})+2\hat{F}(\o_1\eta_1+\o_2\eta_2+\o_3\eta_3);\label{3.57}\\
d\eta_1&=-{\o}_1(\hat{F}Y+{\hat
Y})+\Omega_{12}\eta_2+L\o_2\eta_3+\mu\o_2\xi_1+\mu\o_1\xi_2;\\
d\eta_2&=-{\o}_2(\hat{F}Y+{\hat
Y})-\Omega_{12}\eta_1-L\o_1\eta_3+\mu\o_1\xi_1-\mu\o_2\xi_2;\\
d\eta_3&=-{\o}_3(\hat{F}Y+{\hat
Y})-L\o_2\eta_1+L\o_1\eta_2;\\
d\xi_{1}&=-\mu\o_2\eta_1-\mu\o_1\eta_2+\theta_{12}\xi_2;\\
d\xi_{2}&=-\mu \o_1\eta_1+\mu\o_2\eta_2-\theta_{12}\xi_1;\label{3.62}\\
d({\hat{F}Y-\hat Y})&=-\o (\hat{F}Y-{\hat Y}).  \label{3.63}
\end{align}
So $\mathrm{Span}\{\hat{F}Y-\hat Y\}$ is parallel along $M^3$, as well as its orthogonal complement
\[
\mathbb{V}^6=\mathrm{Span}\{\hat{F}Y+\hat Y, \eta_1, \eta_2, \eta_3, \xi_1, \xi_2\}~.
\]
That means both of them are fixed subspaces of $\mathbb{R}^7_1$.
The type of the inner product restricted on these subspaces
depends on the sign of $\hat{F}$, which will not change on a connected open set, because $\hat{F}$ satisfies a linear PDE \eqref{3.50}. We discuss them case by case.

 {\bf Case 1:} $\hat{F}>0$.
This case $\mathbb{V}^6$ is a fixed space-like subspace orthogonal to a fixed time-like line $\mathbb{V}^\perp$. Define
\[
f=\frac{\hat{F}Y+\hat Y}{\sqrt{2\hat{F}}},\qquad   g=\frac{\hat{F}Y-\hat Y}{\sqrt{2\hat{F}}},
\]
which satisfy $\langle f,f\rangle=1,\langle g, g\rangle=-1$.
So $g \in \mathbb{V}^\perp$ is a constant time-like vector, and $f: M^3 \rightarrow \mathbb{S}^5 \subset \mathbb{V}$ is a submanifold in the sphere.

Assume $g=(1,\vec{0})$. This is without loss of generality since we can always apply a Lorentz transformation in $\mathbb{R}^{7}_{1}$ to $Y$ and its frame at the beginning if necessary, whose effect on $x(M)\subset\mathbb{S}^5$ is a
M\"obius transformation. Then from the geometric meaning of the mean curvature sphere $\xi_r$ explained in Remark~\ref{rem-xi}, we know that $x: M^3\rightarrow \mathbb \mathbb{S}^{5}$ is a minimal Wintgen ideal submanifold (up to a conformal transformation).

Note that this special minimal submanifold can now be identified with $f$ since we have
\[
f+g=\sqrt{2\hat{F}}Y=\sqrt{2\hat{F}}\rho (1,x).
\]
Comparison shows $\sqrt{2\hat{F}}\rho=1$ and $x=f$.
That $f:M^3 \rightarrow \mathbb{S}^5 \subset \mathbb{V}$ is minimal can be verified directly by \eqref{3.57}$\sim$\eqref{3.62}.

{\bf Case 2:} $\hat{F}<0$.
In this case, $\mathbb{V}^6$ is a fixed Lorentz subspace
orthogonal to a constant space-like line $\mathbb{V}^\perp$. Define $f=(\hat{F}Y+\hat Y)/\sqrt{-2\hat{F}}, g=(\hat{F}Y-\hat Y)/\sqrt{-2\hat{F}}.$
Then similar to Case~1 we know $x$ is M\"obius equivalent to
$f: M^3 \rightarrow \mathbb{H}^5 \subset \mathbb{V}\cong \mathbb{R}^6_1$ which is a minimal Wintgen ideal submanifold.

{\bf Case 3:} $\hat{F}\equiv 0$.
Now $\hat \o_i \equiv0, ~i=1,2,3$. So $d\hat Y= -\o\hat{Y}$.
That means $\hat Y$ determines a constant light-like direction. From $d\o=0$, we get $w=d\tau$ for some locally defined function $\tau$. Up to a Lorentz transformation one may take $e^{\tau}\hat Y=(-1, 1, \vec{0})$ which is still denoted by $g$. Since $\langle \xi_r,g\rangle=0$, from the geometric meaning of the mean curvature sphere $\xi_r$ explained in Remark~\ref{rem-xi2} we know that $x$ is a three dimensional minimal Wintgen ideal submanifold in $\mathbb{R}^5$ (up to a suitable conformal transformation).\\

Finally we show (3) implies (1), i.e., for
any minimal Wintgen ideal submanifold $x: M^3 \rightarrow \mathbb{Q}^5(c)$ whose distribution $\mathbb{D}=\mathrm{Span}\{E_1, E_2\}$ is not integrable, there is always $d\o=0$.

By assumption, for this $x$ we can always take local orthonormal frame $\{e_1, e_2, e_3\}$ and $\{n_1, n_2\}$ for the tangent and normal bundles, such that the second fundamental form is given by
\begin{equation}
h^1=\begin{pmatrix}
0& \nu&0\\
\nu&0&0\\
0&0&0
\end{pmatrix}, \qquad
h^2=\begin{pmatrix}
\nu&0&0\\
0&-\nu&0\\
0&0&0
\end{pmatrix}.\label{h}
\end{equation}
It follows that $\rho^2=6\nu^2$ by \eqref{2.2}.
From Remark~2.2, using \eqref{2.23} we always have
\begin{equation}
C^1_1=-C^2_2=-\frac{e_2(\nu)}{6\nu^2}, ~~~~~~~C_1^2=C^2_1=-\frac{e_1(\nu)}{6\nu^2}.
\end{equation}
Consider the M\"obius position vector $Y: M^3 \rightarrow \mathbb{R}^7_1$ of $x$ defined by \eqref{2.2}. For the M\"obius metric
\[
\mathrm{g}=\langle dY, dY\rangle=\rho^2 dx^2=6\nu^2 dx^2,
\]
we can choose $\{E_1=\frac{e_1}{\sqrt{6}\nu}, \;E_2=\frac{e_2}{\sqrt{6}\nu}, \;E_3=\frac{e_3}{\sqrt{6}\nu}\}$ as a set of local orthonormal basis for $(M^3, \mathrm{g})$ with the dual basis $\{\o_1, \;\o_2, \;\o_3\}$.
By \eqref{omega},
\[
2\o_{12}+\theta_{12}=-U\o_1-V\o_2+L\o_3,
\]
We have
\[
C^1_{1,1}=E_1(C^1_1)+C^1_2(\o_{21}+\theta_{21})(E_1)
=\frac{1}{\sqrt{6}\nu}(-E_1(E_2(\nu))+E_1(\nu)\o_{21}(E_1)),
\]
\[
C^1_{2,2}=E_2(C^1_2)+C^1_1(\o_{12}+\theta_{12})(E_2)
=\frac{1}{\sqrt{6}\nu}(-E_2(E_1(\nu))+E_2(\nu)\o_{12}(E_2)).
\]
Using \eqref{omega} we get
\begin{equation*}
C^1_{1,1}-C^1_{2,2}=2\frac{E_3(\nu)}{\nu} \mu L.
\end{equation*}
So we get that $\frac{G}{L}=\frac{E_3(\nu)}{\nu}$. Since
\begin{equation}
\o=-\sqrt{6}(C^1_2\o_1+C^1_1\o_2)+\frac{G}{L}\o_3
=\frac{E_1(\nu)}{\nu}\o_1+\frac{E_2(\nu)}{\nu}\o_2
+\frac{E_3(\nu)}{\nu}\o_3,\label{3.77}
\end{equation}
it is obvious that $\o$ is an exact 1-form, and $d\o=0$.
This finishes the proof.
\end{proof}
\begin{remark}
In the proof of (1)$\Rightarrow$(3), there is always a constant vector $g$ orthogonal to $\mathrm{Span}\{\xi_1,\xi_2,\eta_1,\eta_2\}$ in either of the three cases. Thus in the foliation described
in (3) of Theorem~\ref{thm-envelop}, each leave is now a geodesic in the corresponding space form. In other words, they are ruled submanifolds. This fact is already known in the study of
austere submanifolds \cite{br},\cite{Dajczer3},\cite{Lu}.
\end{remark}

\section{Two M\"obius characterization results}

In this section we will give two characterization theorems
(in terms of M\"obius invariants) related with the following
minimal Wintgen ideal submanifolds in $\mathbb{S}^5$.

\begin{example}\label{ex1}
Let $\gamma:N^2\to \mathbb{C}P^2$ be a holomorphic curve, and $\pi: \mathbb{S}^5 \rightarrow \mathbb{C}P^2$ be the Hopf fibration.
Then the circle bundle $M^3 \subset \mathbb{S}^5$ over $N^2$ obtained by taking the Hopf fibers over $\gamma(N^2)$ is a three dimensional minimal Wintgen ideal submanifold in $\mathbb{S}^5$
as pointed out in \cite{Smet} (Example~6).
\end{example}

Observe that $S^1$ acts by isometry on $\mathbb{S}^5$
whose orbits give the Hopf fibration.
Thus for $M^3$ as above it has an induced $S^1$ symmetry.
Consider the second fundamental forms given in \eqref{h}; the
invariant $\nu$ must be a constant along every orbit of this $S^1$ action, which is exactly an integral curve of $E_3$
(see \cite{Smet} for details where they use $\xi$ to denote this $E_3$). So we have $G=E_3(\nu)=0$ in this special case. Since they are minimal, by Theorem~\ref{thm-minimal} we have $d\o=0$.
These conditions characterize this class of submanifolds as below.

\begin{theorem}\label{thm-Hopf}
Let $x:M^3\rightarrow\mathbb{S}^5$ be a Wintgen ideal submanifold with non-integrable distribution $\mathbb{D}=\mathrm{Span}\{E_1,E_2\}$. If it satisfies $d\o=0, G=0$, then up to a M\"obius transformation on $\mathbb{S}^5$, $x$ is the Hopf lift of a
holomorphic curve given in Example~\ref{ex1}.
\end{theorem}
\begin{proof}
When $G=0$, comparing the coefficients of $\o_1\wedge\o_3$ in \eqref{compare3} and using \eqref{compare5} yields $2\hat{F}=L^2$.
From the proof to theorem~\ref{thm-minimal} we know that $x$ is M\"obius equivalent to a minimal Wintgen ideal submanifold
\[
f=\frac{\hat{F}Y+\hat Y}{\sqrt{2\hat{F}}}: M^3\rightarrow \mathbb{S}^5 \subset \mathbb{R}^6,
\]
where $\mathbb{R}^6=\mathrm{Span}_{\mathbb{R}}
\{f,\eta_3,\eta_1,\eta_2,\xi_1,-\xi_2\}$.
Using \eqref{3.7}$\sim$\eqref{3.10} and $2\hat{F}=L^2, d\hat{F}=-2\o\hat{F}$, with respect to this frame
we can write out the structure equations of $f$:
\begin{equation}\label{Theta1}
d\begin{pmatrix}
f\\ \eta_3\\ \eta_1\\ \eta_2\\ \xi_1\\ -\xi_2\end{pmatrix}=\begin{pmatrix}
0& L\o_3& L\o_1& L\o_2& 0& 0\\
-L\o_3& 0& -L\o_2& L\o_1& 0& 0\\
-L\o_1& L\o_2& 0& \Omega_{12}& \mu\o_2& -\mu\o_1\\
-L\o_2& -L\o_1& -\Omega_{12}& 0& \mu\o_1& \mu\o_2\\
0& 0& -\mu\o_2& -\mu\o_1& 0& -\theta_{12}\\
0& 0&  \mu\o_1& -\mu\o_2& \theta_{12}& 0
\end{pmatrix}\begin{pmatrix}
f\\ \eta_3\\ \eta_1\\ \eta_2\\ \xi_1\\ -\xi_2\end{pmatrix}.
\end{equation}
Denote the frame as a matrix $T:M^3\to \mathrm{SO}(6)$ with respect to a fixed basis $\{{\bf e}_k\}_{k=1}^6$ of
$\mathbb{R}^6$, we can rewrite \eqref{Theta1} as
\begin{equation}\label{Theta2}
dT=\Theta T.
\end{equation}
The algebraic form of $\Theta$ motivates us to introduce a complex structure ${\bf J}$ on $\mathbb{R}^6= \mathrm{Span}_\mathbb{R}\{f,\eta_3,\eta_1,\eta_2,\xi_1,\xi_2\}$ as below:
\[
{\bf J}\begin{pmatrix}f\\
\eta_3\\ \eta_1\\ \eta_2\\ \xi_1\\ -\xi_2\end{pmatrix}
=\begin{pmatrix}
\begin{pmatrix}0& -1\\1 & 0\end{pmatrix} & & \\
& \begin{pmatrix}0& -1\\1 & 0\end{pmatrix} & \\
& &\begin{pmatrix}0& -1\\1 & 0\end{pmatrix}
\end{pmatrix}\begin{pmatrix}
f\\ \eta_3 \\ \eta_1 \\ \eta_2 \\ \xi_1 \\ -\xi_2
\end{pmatrix}.
\]
Denote the diagonal matrix at the right hand side as $J_0$. Then the matrix representation of operator ${\bf J}$ under
$\{{\bf e}_k\}_{k=1}^6$ is:
\[
J=T^{-1}J_0T.
\]
Using $dT=\Theta T$ and the fact that $J_0$ commutes with $\Theta$, it is easy to verify
\[
dJ=-T^{-1}dT T^{-1}J_0T+T^{-1}J_0 dT
=-T^{-1}\Theta J_0T+T^{-1}J_0 \Theta T=0.
\]
So ${\bf J}$ is a well-defined complex structure on
this $\mathbb{R}^6$.

Another way to look at the structure equations \eqref{Theta1}
is to consider the complex version:
\begin{align}
d(f+i\eta_3)&=-iL\o_3(f+i\eta_3)+L(\o_1-i\o_2)(\eta_1+i\eta_2), \label{5.1}\\
d(\eta_1+i\eta_2)&=-L(\o_1+i\o_2)(f+i\eta_3)-i\Omega_{12}(\eta_1+i\eta_2)
+i\mu(\o_1-i\o_2)(\xi_1-i\xi_2),
\notag\\
d(\xi_1-i\xi_2)&=i\mu(\o_1+i\o_2)(\eta_1+i\eta_2)
+i\theta_{12}(\xi_1-i\xi_2).\label{5.3}
\end{align}
Geometrically, this implies that
\[
\mathbb{C}^3=\mathrm{Span}_{\mathbb{C}}\{f+i\eta_3,
\eta_1+i\eta_2,\xi_1-i\xi_2\},
\]
is a fixed three dimensional complex vector space endowed with the complex structure $i$, which is identified with $(\mathbb{R}^6,{\bf J})$ via the following isomorphism between
complex linear spaces:
\[
v\in \mathbb{C}^3~~\mapsto~~\mathrm{Re}(v)\in \mathbb{R}^6.
\]
For example, $f+i\eta_3\mapsto f,if-\eta_3\mapsto -\eta_3$ and
so on.

The second geometrical conclusion is an interpretation of \eqref{5.1} that $[f+i\eta_3]$
defines a holomorphic mapping from the quotient surface $\overline{M}=M^3/\Gamma$ to the projective plane $\mathbb{C}P^2$ (like the conclusion (1) in Theorem~\ref{thm-envelop}).
Moreover, the unit circle in
\[
\mathrm{Span}_{\mathbb{R}}\{f,\eta_3\}=\mathrm{Span}_{\mathbb{R}}\{f,{\bf J}f\}=\mathrm{Span}_{\mathbb{C}}\{f+i\eta_3\}
\]
is a fiber of the Hopf fibration of $\mathbb{S}^5\subset (\mathbb{R}^6,{\bf J})$. It corresponds to the subspace
$\mathrm{Span}_{\mathbb{R}}\{Y,\hat{Y},\eta_3\}$, which is geometrically
a leave of the foliation $(M^3,\Gamma)$ as described by conclusion (3) in Theorem~\ref{thm-envelop}.
Thus the whole $M^3$ is the Hopf lift of $\overline{M}\to\mathbb{C}P^2$.
In other words we have the following commutative diagram
\begin{equation*}
\begin{xy}
(30,30)*+{M^3}="v1", (60,30)*+{\mathbb{S}^5}="v2", (90,30)*+{\mathbb{C}^3}="v3";%
(30,0)*+{\overline{M}}="v4", (60,0)*+{\mathbb{C}P^2}="v5".%
{\ar@{->}^{f} "v1"; "v2"}%
{\ar@{->}^{\subset} "v2"; "v3"}%
{\ar@{->}_{M^3/\Gamma} "v1"; "v4"}%
{\ar@{->}_{[f+i\eta_3]} "v1"; "v5"}%
{\ar@{->}^{} "v4"; "v5"}%
{\ar@{->}^{\pi} "v2"; "v5"}%
{\ar@{->}_{\pi} "v3"; "v5"}%
\end{xy}
\end{equation*}
This finishes the proof.
\end{proof}

Among examples given above, there is a special one coming from
the lift of the famous Veronese embedding $\gamma:\mathbb{C}P^1\to \mathbb{C}P^2$ which is homogeneous.
Thus the lift $M^3$ is itself a homogeneous minimal Wintgen ideal submanifold in $\mathbb{S}^5$.
This special example can also be described as below.
\begin{example}\label{ex2}
The orthogonal group $\mathrm{SO}(3)$ embedded in $\mathbb{S}^5$ homogeneously:
\begin{equation}\label{SO3}
x: \quad \mathrm{SO}(3)~\rightarrow~ \mathbb{S}^5, \qquad
(u, v, u\times v)  \mapsto \frac{1}{\sqrt{2}}(u, v).
\end{equation}
The orthonormal frames of the tangent and normal bundles can be chosen as
$
e_1=(u\times v, 0), e_2=(0, u\times v), e_3=\frac{1}{\sqrt{2}}(-v, u);~~
n_1=\frac{-1}{\sqrt{2}}(v, u), n_2=\frac{-1}{\sqrt{2}}(u, -v).
$
Direct computation verifies that it is a minimal Wintgen ideal submanifold.
\end{example}

Consider the canonical lift $Y=\sqrt{6}(1,x): \mathrm{SO}(3)\rightarrow \mathbb{R}^7_1.$
The M\"obius metric is given by $\mathrm{g}=6dx\cdot dx.$ It follows from \eqref{2.23} that the M\"obius form vanishes, i.e.,
$C^{r}_{j}=0$. Next,
$\{E_j=e_j/\sqrt{6}\}$ form an orthonormal frame for $(M^3, \mathrm{g})$, with the dual 1-form $\{\o_j\}$. So the frame used in Section~3 is given by
\[\eta_j=Y_j=(0,e_j), ~~\hat{Y}=N=\frac{1}{2\sqrt{6}}(-1,x), ~~\xi_r=(0,n_r).\]
The structure equations are
\begin{equation}
d\begin{pmatrix}
Y\\ \hat{Y} \\ \eta_1 \\ \eta_2 \\ \eta_3 \\ \xi_1 \\ \xi_2
\end{pmatrix}=\begin{pmatrix}
0&0&\o_1&\o_2&\o_3&0&0\\
0&0&\frac{\o_1}{12}&\frac{\o_2}{12}&\frac{\o_3}{12}&0&0\\
\frac{-\o_1}{12}&-\o_1&0&0&\frac{\o_2}{\sqrt{6}}&
\frac{\o_2}{\sqrt{6}}&\frac{\o_1}{\sqrt{6}}\\
\frac{-\o_2}{12}&-\o_2&0&0&\frac{-\o_1}{\sqrt{6}}&
\frac{\o_1}{\sqrt{6}}&\frac{-\o_2}{\sqrt{6}}\\
\frac{-\o_3}{12}&-\o_3&\frac{-\o_2}{\sqrt{6}}&
\frac{\o_1}{\sqrt{6}}&0&0&0\\
0&0&\frac{-\o_2}{\sqrt{6}}&\frac{-\o_1}{\sqrt{6}}&0&0&
\frac{\o_3}{\sqrt{6}}\\
0&0&\frac{-\o_1}{\sqrt{6}}&\frac{\o_2}{\sqrt{6}}&0&
\frac{-\o_3}{\sqrt{6}}&0
\end{pmatrix}\begin{pmatrix}
Y\\ \hat{Y} \\ \eta_1 \\ \eta_2 \\ \eta_3 \\ \xi_1 \\ \xi_2
\end{pmatrix}.\label{4.15}
\end{equation}

In \cite{Smet} they gave a characterization of this example as
the unique Wintgen ideal submanifold $M^m\to\mathbb{Q}^{m+2}(c)$
with constant non-zero normal curvature.
Here we provide another characterization of it in M\"obius geometry.

In the statement below, a connected submanifold $M$ in $\mathbb{S}^5$ is said to be \emph{locally M\"obius homogenous} if
for any two points $p,q\in M$, there are two neighborhoods
$U_p,U_q\subset M$ of them respectively and a M\"obius transformation $T$ such that $T(p)=q, T(U_p)=U_q$.
An essential property of a locally (M\"obius) homogenous submanifold is that any well-defined (M\"obius) invariant function on it must be a constant.

\begin{theorem}\label{thm-homog}
Let $x: M^3\rightarrow \mathbb{S}^5$ be a Wintgen ideal submanifold of dimension 3. If it is locally M\"obius homogenous and the distribution $\mathbb{D}=\mathrm{Span}\{E_1, E_2\}$ is not integrable, then up to a M\"obius transformation this is part of $x: \mathrm{SO}(3) \rightarrow \mathbb{S}^5$ given in Example~\ref{ex2}.
\end{theorem}
\begin{proof}
According to Remark~\ref{rem-transform} and Remark~\ref{rem-invariant}, the coefficients $\{\hat{F}, G, L\}$ appearing in the structure equations are geometric invariants.
Thus under our assumption these functions must be constants;
in particular $L$ is a non-zero constant.

By \eqref{compare4} we know $G=E_3(L)=0$. From \eqref{compare1}$\sim$\eqref{compare5}, we have
\begin{equation}
\qquad 2\hat{F}=L^2,~~~
\hat\o_3=LU\o_1+LV\o_2+\hat{F}\o_3.\label{5.7}
\end{equation}
By \eqref{Omega12+} one can write out explicitly that
\begin{align}
\Omega_{12}&=\alpha\o_1+\beta\o_2+\gamma\o_3.\label{5.8}\\
\theta_{12}&=(U-2\alpha)\o_1+(V-2\beta)\o_2+(L-2\gamma)\o_3.
\label{5.9}
\end{align}
Note that in general $\alpha, \beta, \gamma$ are not geometric invariants, because they are components of the connection 1-form $\Omega_{12}$, and when the frame $\{E_1,E_2\}$ rotate by angle $t$ in \eqref{transform}, $\Omega_{12}$ will differ by a closed
1-form $dt$.

On the other hand, by \eqref{UVL}, $U^2+V^2$ is the square of the norm of the M\"obius form $\Phi$ (up to a non-zero constant), hence a geometric invariant. It must also be a constant on $M^3$.\\

\textbf{Claim}: The 1-form $\o=0$ identically; i.e., $U=V=0$ on $M^3$ everywhere and under any frame $\{E_1,E_2\}$. (As a consequence of this fact and the conclusion of Theorem~\ref{thm-Hopf}, any of such examples is the Hopf lift of a complex curve in $\mathbb{C}P^2$.)\\

We prove this claim by contradiction. Suppose $U^2+V^2$ is a
non-zero constant. We can choose a canonical frame according to Remark~\ref{rem-invariant}.
With respect to such a canonical frame,
all coefficients $\alpha,\beta,\gamma$ in \eqref{5.8} are now well-defined functions, hence be constants.
From \eqref{3.14}$\sim$\eqref{3.21} we have
\begin{align}
d\Omega_{12}=&(\alpha^2-\alpha U+\beta^2-\beta V+2L\gamma)\o_1\wedge\o_2  \notag \\
&+[\alpha(L-\gamma)+\gamma U]\o_2\wedge\o_3
-[\beta(L-\gamma)+\gamma V]\o_1\wedge\o_3,\label{dOmega12}\\
d\theta_{12}=&-[(U-2\alpha)(U-\alpha)
+(V-2\beta)(V-\beta)-2(L-2\gamma)L]\o_1\wedge\o_2  \notag\\
&+[(U-2\alpha)(L-\gamma)+U(L-2\gamma)]\o_2\wedge\o_3 \notag\\
&-[(V-2\beta)(L-\gamma)+V(L-2\gamma)]\o_2\wedge\o_3. \label{dtheta12}
\end{align}
Comparing the coefficients with \eqref{3.18} and \eqref{3.24}
separately, we obtain
\begin{gather}
\alpha^2-\alpha U+\beta^2-\beta V+2L\gamma=2\mu^2-2L^2, \label{5.10}\\
\alpha(L-\gamma)+\gamma U=0,\label{5.11}\\
\beta(L-\gamma)+\gamma V=0, \label{5.12}\\
2\mu^2=-(U-2\alpha)(U-\alpha)-(V-2\beta)(V-\beta)
+2(L-2\gamma)L,\label{5.13}\\
(U-2\alpha)(L-\gamma)+U(L-2\gamma)=0,\label{5.14}\\
(V-2\beta)(L-\gamma)+V(L-2\gamma)=0. \label{5.15}
\end{gather}
If $\gamma=L\ne 0$, then from \eqref{5.11} and \eqref{5.12} we have $U=V=0$. This does not only contradict with the assumption
$U^2+V^2\ne 0$, but also implies from \eqref{5.13} that $\mu^2=-\alpha^2-\beta^2-L^2$, a contradiction with $L\ne 0,\mu=1/\sqrt{6}\ne 0$.

If $\gamma\ne L$, then \eqref{5.11} and \eqref{5.12} tell us
$\alpha=-\frac{\gamma }{L-\gamma}U,\beta=-\frac{\gamma }{L-\gamma}V.$
Combined with \eqref{5.14} and \eqref{5.15}, we get
$\gamma=2L, \alpha=2U, \beta=2V.$ Insert them into \eqref{5.13}, we have $2\mu^2=-3U^2-3V^2-6L^2.$
So there must be $U=V=L=0$, which also contradicts to our assumption. Thus the claim is proved.\\

Now that $U=V=\frac{G}{L}=0$, we have $\o=0$ and
$2\Omega_{12}+\theta_{12}=L\o_3$. Differentiate the last equation. We get
\[
d(2\Omega_{12}+\theta_{12})=Ld\o_3=2L^2\o_1\wedge\o_2.
\]
On the other hand, still by \eqref{3.18}\eqref{3.24} there is
\[
d(2\Omega_{12}+\theta_{12})=(6\mu^2-4L^2)\o_1\wedge\o_2.
\]
Comparison shows
\[L=\mu=1/\sqrt{6}
\]
(assume $L>0$ without loss of generality). Then by \eqref{Omega12}\eqref{3.18} and $2\hat{F}=L^2$ in \eqref{5.7}, we obtain
\[
d\Omega_{12}=d\o_{12}=2(\mu^2-L^2)\o_1\wedge\o_2=0.
\]
Thus $\Omega_{12}$ is a closed 1-form, which is locally an exact 1-form. Then we can use
\eqref{transform} to find another frame such that $\Omega_{12}=\o_{12}=0$ and such frame is canonically chosen once
it is fixed at an arbitrary point. With respect to this frame on a simply connected domain of $M^3$ we know $\alpha=\beta=\gamma=0$ in \eqref{5.8},
and $\theta_{12}=\frac{1}{\sqrt{6}}\o_3$ in \eqref{5.9}. The proof is finished by checking that the structure equations are the same as \eqref{4.15} for $x(\mathrm{SO}(3))$.
\end{proof}

\begin{remark}
For M\"obius homogeneous Wintgen ideal submanifolds $M^3\to \mathbb{S}^5$ with $L=0$ (integrable $\mathbb{D}$), the classification will not be difficult. By the conclusion of
Theorem~A in the introduction, such an example comes from super-minimal surface $\overline{M}$ in four dimensional space forms. This super-conformal $\overline{M}$ must also be homogeneous by itself. According to our classification of
Willmore surfaces with constant M\"obius curvature \cite{MaWang}, this $\overline{M}$ should be the Veronese surface $\mathbb{R}P^2\to \mathbb{S}^4$,
and the original $M^3$ is a cone in $\mathbb{R}^5$ over this surface.
\end{remark}

\vspace{5mm} \noindent Zhenxiao Xie,
{\small\it School of Mathematical Sciences, Peking University,
Beijing 100871, People's Republic of China.
e-mail: {\sf xiezhenxiao@126.com}}

\vspace{5mm} \noindent Tongzhu Li,
{\small\it Department of Mathematics, Beijing Institute of
Technology, Beijing 100081, People's Republic of China.
e-mail:{\sf litz@bit.edu.cn}}

\vspace{5mm} \noindent Xiang Ma,
{\small\it LMAM, School of Mathematical Sciences, Peking University,
Beijing 100871, People's Republic of China.
e-mail: {\sf maxiang@math.pku.edu.cn}}

\vspace{5mm} \noindent Changping Wang
{\small\it School of Mathematics and Computer Science,
Fujian Normal University, Fuzhou 350108, People's Republic of China.
e-mail: {\sf cpwang@fjnu.edu.cn}}
\end{document}